\RequirePackage{fix-cm}
\documentclass[smallextended]{svjour3}       % onecolumn (second format)
\smartqed  % flush right qed marks, e.g. at end of proof
\usepackage{graphicx}
\graphicspath{{figures/}}
\usepackage{mathptmx}      % use Times fonts if available on your TeX system
\usepackage{natbib}
\usepackage{setspace}
\usepackage{amssymb}
\usepackage{subfig}
\usepackage{caption}
\usepackage{amsfonts}
\usepackage{hyperref}
\hypersetup{
	colorlinks=true,
	linkcolor=blue,
	filecolor=blue,
	citecolor = blue,      
	}
\usepackage{amsmath}

\usepackage{color}
\usepackage{colordvi}
\numberwithin{equation}{section}

\newtheorem{prop}{Proposition}

\newcommand{\komment}[1]{}
 \journalname{}%Journal of Mathematical Biology}
\begin{document}
	\title{An Immuno-epidemiological Model Linking Between-host and Within-host Dynamics of Cholera%
	\thanks{This research is supported by a grant from the German Academic Exchange Service (DAAD) and by the TUM International Graduate School of Science and Engineering (IGSSE), 
		within the project GENOMIE\_QADOP.}
}


\author{Beryl Musundi       
	%  \and
	%Second Author %etc.
}


\institute{Beryl Musundi \at
	Faculty of Mathematics, Technische Universit\"at M\"unchen, 85748 Garching, Germany \\
	Department of Mathematics, Moi university, 3900-30100 Eldoret, Kenya\\
	\email{beryl.musundi@tum.de}           %  \\
}

\date{Received: date / Accepted: date}

\maketitle

\begin{abstract}
Cholera, a severe gastrointestinal infection caused by the bacterium \textit{Vibrio cholerae}, remains a major threat to public health with a yearly estimated global burden of 2.9 million cases. Although the majority of existing models for the disease focus on its population dynamics, it's important to link the multiple scales of the disease to gain better perspectives on its spread and control. In this study, we formulate an immuno-epidemiological model for cholera linking the between-host and within-host dynamics of the disease. The within-host model utilizes time-scale methods to differentiate the pathogen dynamics from the dynamics of the immune response. Bifurcation analysis of the within-host system reveals the necessary conditions for the existence of both the Hopf and saddle-node bifurcations. Contrary to other within-host models, the current approach allows for the elimination of the pathogen after a finite time. The epidemic model takes into account the direct human-to-human transmission route of the infection as well as the transmission via the environment. It is represented by a dynamical system structured on the immune status, which is a function derived from the within-host immune response. The basic reproduction number is derived and the stability of equilibrium points analysed. Analysis of the endemic equilibrium reveals additional constraints that lead to its stability. Without loss of immunity, the endemic equilibrium, if it exists, is globally asymptotically stable. 
	%Include keywords, PACS and mathematical
	%subject classification numbers as needed.
	\keywords{Cholera \and Within-host dynamics \and   Between-host dynamics \and Time-scale analysis \and Stability}
	 \subclass{MSC 92D30 \and MSC 34D15 \and MSC 35Q92}
\end{abstract}

\section{Introduction}
Infectious diseases remain a major cause of human mortality and morbidity despite the advances in medicine \citep{garira2017complete}. A holistic understanding of the transmission dynamics of these diseases is necessary for the development of better approaches aimed at reducing their transmission \citep{hethcote2000mathematics}. Two scales of interactions occur when a host comes into contact with a pathogen. These scales are characterised by the epidemiological process that is linked to disease transmission in the population and the immunological process that relates the viral-cell interaction at the individual host level \citep{feng2012model}. Two modelling approaches have been associated with these processes. The between host approach whose main focus is on the disease dynamics in the population and the within-host approach that looks at the disease from an individual host level \citep{wang2017disease}. The two approaches are frequently used independently as seen in \citet{shuai2011global, wang2017modeling}. However, models with multiple scales that link the between-host and within-host processes provide new perspectives in the host-parasite interactions. Such models, which have gained interest in recent times, are referred to as immuno-epidemiological models \citep{martcheva2015coupling}. These type of models explain the role of the within-host processes in pathogen evolution as well as make predictions of epidemiological quantities such as the reproduction number and disease prevalence \citep{martcheva2015coupling}. The explicit linkage between the two scales is one of the important aspects of setting up multi-scale models. The prominent linking mechanism for within-host models to between-host models is the pathogen load and the pathogen growth rate while the majority of between-host models are linked to within-host models through the transmission rate \citep{childs2019linked}. \citet{feng2013mathematical} links the within-host dynamics to the between-host dynamics of \textit{Toxiplasmi gondii}, an environmentally transmitted disease, through the pathogen load in the environment.\\ The present paper focuses on an immuno-epidemiological model for cholera, an acute gastrointestinal disease caused by the bacterium \textit{Vibrio cholerae}. This disease continues to affect millions of people in countries that lack access to safe water and proper sanitation infrastructure with the global burden estimated at 2.9 million cases and 95,000 deaths \citep{ali2015updated}. Sub-Saharan Africa bears the greatest burden of this disease. Cholera is transmitted directly through human to human contact and indirectly from the environment through contaminated food and water \citep{hartley2005hyperinfectivity}. The dynamics of the disease are therefore largely dependant on the diverse interactions between the environment, the human host and the pathogen \citep{hartley2005hyperinfectivity}. When the bacteria are ingested, they must survive the stomach's gastric acid. They then penetrate the mucus lining of the epithelial cells, colonize them and secrete a Cholera Toxin (CT) that causes the onset of cholera symptoms \citep{reidl2002vibrio}. These symptoms include watery diarrhoea and vomiting. Infected persons are either symptomatic or asymptomatic and can shed the bacteria back to the environment through their stool. Studies have shown that the freshly shed vibrios are more infectious in comparison to environmental vibrios. They are also responsible for the explosive nature of the disease \citep{hartley2005hyperinfectivity}. It is therefore essential to incorporate the within-host dynamics in the epidemic modelling of the disease.\\ 
The bulk of the developed cholera models centre on the epidemic spread of the disease \citep{hartley2005hyperinfectivity,mukandavire2011estimating, shuai2011global, tian2011global, brauer2013dynamics}. A within-host model based on the bacterial-viral interaction of the disease \citep{wang2017modeling} is among the few attempts made at modelling the disease at the within-host level.\\ 
Recent attempts have also been made in the development and analysis of multi-scale cholera models.
A multi-scale model that links the between-host and within-host dynamics of cholera through the concentration of human vibrios is formulated in \citet{wang2017disease}. The between-host dynamics are represented by a SIRS model. An additional environmental compartment outlining the bacterial evolution in the environment is also set up. The dynamics of the within-host model, which describes the growth of human vibrios inside the body are, however, represented very simply by the use of a single ordinary differential equation. Furthermore, the interaction of the pathogen with the immune system is not taken into account. \\
\citet{ratchford2019modeling} subdivide the dynamics of cholera into three subsystems that show the different time scales involved in the growth of a cholera infection. The subsystems represent the within-host, between-host and environmental dynamics. The within-host system models the interaction of the immune system with human vibrios and viruses. The environmental growth of the vibrios provides the linkage of the within-host to the between-host system. The within-host immune response is not considered as a variable in the epidemic model of the disease, consequently, neglecting the effects of immunity.\\
In this paper, we aim to extend the knowledge of multi-scale modelling of cholera by formulating an immuno-epidemiological model that couples the within-host and between-host dynamics of cholera. The within-host dynamics which describe the interaction between healthy cells, the pathogen and immune response are represented by a system of ordinary differential equations. The pathogen is considered to undergo some Allee effects and its dynamics are distinguished from the dynamics of the immune response through the use of time scales. The between-host model describes the spread of the disease in the population and is represented by a size-structured model. In our approach, the immune status, which is a function of the within-host immune response, is considered to be the physiological variable that structures the infected population. Both the environmental and human transmission pathways of the disease are taken into account with the within-host pathogen load providing an additional link between the two systems. We analyse our model and verify the validity of the results through numerical simulations.\\
The remainder of this paper is organised as follows. In section \ref{sec:2} we formulate the within-host model, check for boundedness of solutions and carry out a time-scale analysis of the two subsystems. Furthermore, we perform a bifurcation analysis of the system and simulate the results numerically. In section \ref{sec:3} we formulate the between-host model, compute the equilibrium points, derive the expression for the reproduction number and analyse the stability of the steady states. Finally, we discuss the results and conclude the paper in section \ref{sec:4}.
\section{Within-host model}\label{sec:2}
\subsection{Model Formulation}
The model is a modified version of the within-host models with immune response reviewed in \citet{martcheva2015coupling} that describe the interaction between the pathogens and the immune system. In this case, the within-host model describes the interaction between the target cells, cholera pathogens and the immune response. 
\begin{eqnarray}\label{1}
\frac{dT}{dt} &=& \Lambda - \mu T -\alpha P^2T \nonumber \\
\frac{dP}{dt} &=& \alpha P^2T - \gamma P-\delta P W \\
\frac{dW}{dt} &=& \varepsilon (\kappa P-c W)\nonumber.
\end{eqnarray}
The variables T, P and W represent
the density of target cells, the pathogen load and the immune response respectively. We take the incidence to be quadratic in P, to model the Allee effect. The Allee effect, derived from the work by \citet{allee1927animal}, defines a positive correlation between the population density and population growth rate of some species. Populations with this effect show reduced growth rates at low densities \citep{drake2011allee}. Microbial populations with quorum sensing mechanisms such as \textit{Vibrio fischeri} and \textit{Vibrio cholerae} may exhibit this effect \citep{kaul2016experimental,jemielita2018quorum}. The parameter $\Lambda$  denotes the rate of production of healthy cells, $\mu$ and $\gamma$ denote the natural death rate of healthy cells and the cholera pathogens respectively. $\alpha$ is the rate of infection of healthy cells, $\delta$  is the rate of clearance of the pathogen by the immune response and $\varepsilon\ll 1$ the slow time scale of 
the immune response ($\varepsilon$ is used to distinguish the dynamics of the immune response from the pathogen dynamics).  $\kappa $  denotes the rate of activation of the immune response in the presence of the pathogen and $c$ is the self-deactivation of the immune response.
\subsection{Positivity and Boundedness}
We show that the model is well-posed by showing that the solutions are positive and bounded.
\begin{prop}
	Let all the parameters of system (\ref{1}) be non-negative. A non-negative solution (T(t), P(t), W(t)) exists for all state variables with
	non-negative initial conditions $(T(0)\geq 0, P(0) \geq 0, W(0)\geq 0)$ for all $t\geq0$.
\end{prop}
\begin{theorem}
	The set $\Omega = \{(T,P, W)$
	$\in$ $\mathbb{R}_+^3 ; T+P \leq\frac{\Lambda}{\mu+\gamma}+1, W  \leq\frac{\kappa \Lambda}{c(\mu+\gamma)}+ 1  \}$ is an absorbing set for the system (\ref{1}).
\end{theorem}
\begin{proof}
	Taking the total population $N=T+P$
	\begin{eqnarray*}
		\frac{dN}{dt} &=& \Lambda - \mu T -\alpha P^2T+\alpha P^2T - \gamma P-\delta P W\\
		\frac{dN}{dt} &\leq& \Lambda -(\mu+\gamma) N.
	\end{eqnarray*}
	Solving using the integrating factor and applying initial conditions gives 
	\begin{equation*}
	N(t) \leq \frac{\Lambda}{\mu+\gamma}+ \bigg [N(0)-\frac{\Lambda}{\mu+\gamma}\bigg] e^{-(\mu +\gamma)t}	,
	\end{equation*}
	as $t\rightarrow \infty$ $\lim_{t\to\infty} N(t)\leq\frac{\Lambda}{\mu+\gamma}$
	similarly, as $t \rightarrow 0$ $\lim_{t\to 0}N(t)\leq N(0)$.\\
	This implies that N is eventually bounded by 
	$\frac{\Lambda}{\mu+\gamma}+1$. 
	Using similar arguments, we can also show that W is bounded.\\
	Since the solutions to system (\ref{1}) are positive and bounded, the model is biologically meaningful.\qed
\end{proof}
\subsection{Time-scale Analysis}
The dynamics of the pathogen and target cells are considered to take place at a faster time scale in comparison to the immune response. We, therefore, use time scale analysis to analyse system (\ref{1}).
\subsubsection{Fast System}
For $\epsilon \rightarrow 0$. The fast system is given by $\frac{dW}{dt} =0$ and
\begin{eqnarray}\label{2}
\frac{dT}{dt} &=& \Lambda - \mu T -\alpha P^2T \nonumber\\
\frac{dP}{dt} &=& \alpha P^2T - \gamma P-\delta P W. 
\end{eqnarray}
\begin{prop}\label{a}
	The system (\ref{2}) has a trivial infection free stationary point $E_0=(\frac{\Lambda}{\mu},0)$ which is always locally asymptotically stable and additionally for $\Lambda > 2(\gamma + \delta  W)\sqrt{\frac{\mu}{\alpha}}$  a non trivial stationary point $E_1=( T^*,P^*)$ given by 
	\begin{equation*}
	T^*=  \frac{\Lambda}{\mu +\alpha P^{*2}}, \qquad P^*=\frac{1}{2(\gamma + \delta  W)\alpha}\big[\alpha \Lambda \pm\sqrt{\alpha^2 \Lambda^2-4\alpha \mu (\gamma + \delta  W)^2}\big].
	\end{equation*}   
\end{prop}
\begin{proof}
	To find the equilibrium points we set the right hand side of system (\ref{2}) to  zero
	\begin{eqnarray}\label{3}
	\Lambda - \mu T -\alpha P^2T &=& 0 \nonumber \\
	\alpha P^2T-\gamma P-\delta P W &=& 0.	
	\end{eqnarray}  
	At the trivial equilibrium point $P=0$ and $T= \frac{\Lambda}{\mu}$ thus $E_0=(\frac{\Lambda}{\mu},0)$.\\ 
	Linearizing system (\ref{2})  gives the Jacobian matrix 
	\begin{equation}
	J=
	\left[ {\begin{array}{cc}\label{4}
		-\mu -\alpha P^2 & -2\alpha P T \\
		\alpha P^2 & 2\alpha PT-\gamma -\delta W \\
		\end{array} } \right].
	\end{equation}
	At $E_0$ matrix (\ref{4}) is given by 
	\[J=
	\left[ {\begin{array}{cc}
		-\mu  & 0\\
		0 & -\gamma -\delta W \\
		\end{array} } \right]
	.\]
	Its characteristic equation has two negative roots and thus the stationary point is locally asymptotically stable.\\
	For the non-trivial equilibrium point, we solve for $T^*$ from the first equation of (\ref{3}) to get
	$T^*=  \frac{\Lambda}{\mu +\alpha P^{*2}}$.
	Substituting $T^*$ in the second equation gives
	\begin{equation*}
	P^*(	\alpha P^* (\frac{\Lambda}{\mu +\alpha P^{*2} })-\gamma- \delta W)=0.
	\end{equation*}
	Since $P^*\ne0$
	\begin{eqnarray}
	\alpha \Lambda P^*&=&(\gamma+ \delta W)(\mu +\alpha P^{*2}),\nonumber\\ 
	\alpha (\gamma +\delta W) P^{*2}-\alpha \Lambda P^* + (\gamma + \delta  W)\mu&=&0,\label{5}
	\end{eqnarray} 
	\begin{equation}\label{6}
	P^*=\frac{1}{2(\gamma + \delta  W)\alpha}\big[\alpha \Lambda \pm\sqrt{\alpha^2 \Lambda^2-4\alpha \mu (\gamma + \delta  W)^2}\big].
	\end{equation}
	$P^*$ exists whenever
	\begin{equation*}
	\alpha^2 \Lambda^2>4\alpha \mu (\gamma + \delta  W)^2 \qquad\implies\qquad
	\Lambda>2(\gamma + \delta  W)\sqrt{\frac{m}{\alpha}}. 
	\end{equation*}\qed
\end{proof}
\subsubsection*{Bifurcation Analysis}
\subsubsection*{Saddle-node Bifurcation}
A saddle-node bifurcation occurs when two stationary points of a dynamical system collide and annihilate each other. 
Proposition (\ref{a}) indicates that we have a saddle-node bifurcation whenever 
\begin{eqnarray*}
	\Lambda&=&2(\gamma + \delta  W)\sqrt{\frac{m}{\alpha}} .
\end{eqnarray*}
\subsubsection*{Hopf Bifurcation}
A Hopf bifurcation occurs when the system loses stability and periodic orbits appear. It is associated with the occurrence of purely imaginary eigen values \citep{kuznetsov2013elements}. 
\begin{prop}
	System (\ref{2}) has a Hopf point whenever
	\begin{eqnarray*}
		\mu &=& \Gamma-\frac{\Gamma^4}{\alpha \Lambda^2},\\
		\Gamma&>& \mu 
	\end{eqnarray*} where $\Gamma=\gamma + \delta  W$
\end{prop}
\begin{proof}
	From the second equation in (\ref{3}) and equation 
	(\ref{5}) we get
	\begin{equation}\label{7}
	P^*T^* = \frac{\gamma + \delta  W}{\alpha}, \qquad
	\alpha P^{*2}= \frac{\alpha \Lambda P^*}{\gamma + \delta  W}-\mu.
	\end{equation}
	Substituting the two equations in matrix (\ref{4}) and simplifying gives	
	\begin{equation}\label{8}
	J=
	\left[ {\begin{array}{cc}
		-\frac{\alpha \Lambda P^*}{\gamma + \delta  W} & -2(\gamma + \delta  W) \\
		\frac{\alpha \Lambda P^*}{\gamma + \delta  W}-\mu& \gamma + \delta  W \\
		\end{array} } \right],
	\qquad
	J=\frac{\alpha \Lambda}{\gamma + \delta  W}
	\left[ {\begin{array}{cc}
		-P^* & \frac{-2(\gamma + \delta  W)^2}{\alpha \Lambda} \\
		P^*-\frac{\gamma + \delta  W}{\alpha \Lambda}& \frac{(\gamma + \delta  W)^2}{\alpha \Lambda}  \\
		\end{array} } \right].
	\end{equation}
	For the occurrence of a Hopf point the trace of matrix (\ref{8}) should be equal to zero which implies that
	\begin{equation*}
	P^*=\frac{(\gamma + \delta  W)^2}{\alpha \Lambda}.
	\end{equation*}
	Letting $\Gamma=\gamma + \delta  W $ we get	$P^*=\frac{\Gamma^2}{\alpha \Lambda}$ which is substituted in (\ref{5}) to give 
	\begin{equation}\label{9}
	\alpha \Gamma (\frac{\Gamma^2}{\alpha \Lambda})-\alpha \Lambda \frac{\Gamma^2}{\alpha \Lambda} + \mu \Gamma=0 \quad \implies
	\mu=\Gamma -\frac{\Gamma^4}{\alpha \Lambda^2}. 
	\end{equation}
	Substituting this value of $\mu$ in (\ref{6}) gives
	\begin{equation*}
	P^*=\frac{1}{2\Gamma \alpha}\big[\alpha \Lambda \pm\sqrt{\alpha^2 \Lambda^2-4\alpha(\Gamma -\frac{\Gamma^4}{\alpha \Lambda^2})  \Gamma^2}\big].
	\end{equation*}
	Thus
	\begin{equation*}
	P_1^* = \frac{\Gamma^2}{\alpha \Lambda},\qquad
	P_2^*=\frac{\Lambda}{ \Gamma}-\frac{\Gamma^2}{\alpha \Lambda}.
	\end{equation*}
	$P_1^*$ is the required value since it gives us a zero value when substituted in the trace.\\
	Further to the trace being zero the determinant of matrix (\ref{8}) should  also be positive, which implies that
	\begin{equation*}
	-\mu \alpha P^* T^* + \alpha^2 P^{*3} T>0 \qquad
	\implies \qquad	\alpha P^{*2} > \mu.
	\end{equation*}
	Substituting the values of $P_1^*$ and $\mu$ into the above equation gives us	$\Gamma> 2\mu$.\qed
\end{proof}
\subsubsection*{Numerical Simulations}
We perform numerical simulations using the XPPAUT software \citep{ermentrout2002simulating} to confirm the validity of our analytical results. We assume the initial conditions to be; $ T = 0.5$, $P = 0.9$, $\alpha = 1$, $\mu = 0.1$, $\Lambda = 1$, $\gamma=0.5$, $\delta = 0.3$, $W = 0.9$. We take $\delta$ and W to be the bifurcation parameters. In Fig (\ref{fig:f1}) a saddle node bifurcation occurs when $\delta=1.201$ and a Hopf birfucation occurs when $\delta=0.5158$. In Fig (\ref{fig:f2}) a Hopf bifurcation occurs when $W=1.547$. There's is also a possibility of the occurrence of a homoclinic bifurcation in Fig (\ref{fig:f1}) at the point where the periodic orbits collide with the saddle-node. 
\begin{figure}[h]
	\centering
	\subfloat[Saddle node (SN)  and Hopf (H) bifurcation]{\includegraphics[width=0.5\textwidth]{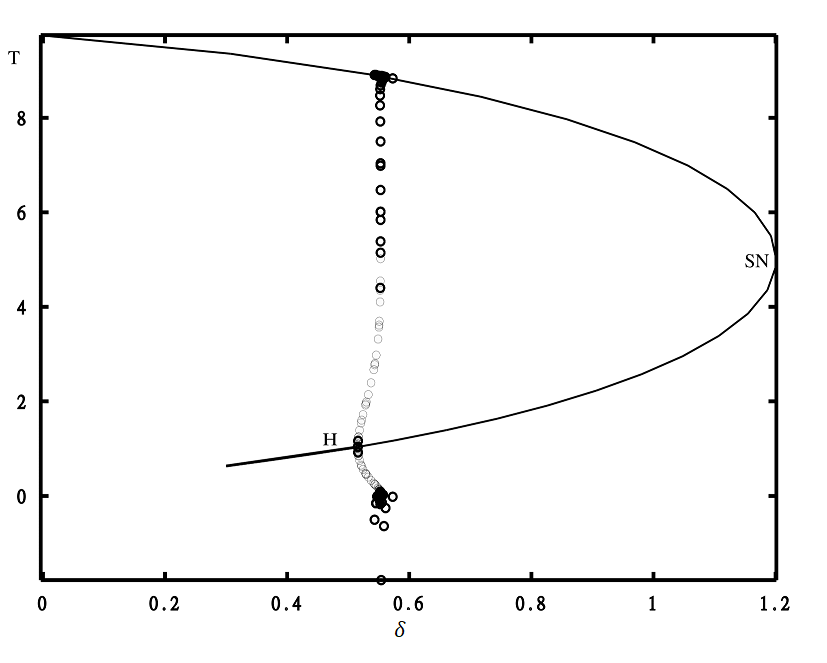}\label{fig:f1}}
	\hfil
	\subfloat[Hopf birfurcation]{\includegraphics[width=0.5\textwidth]{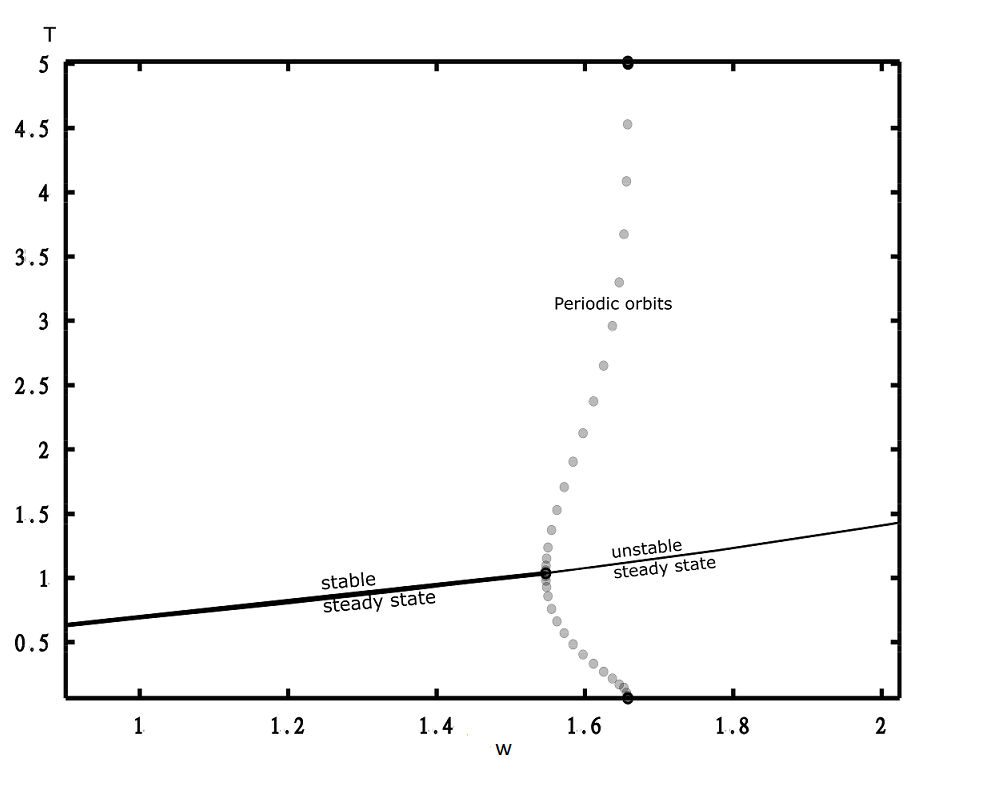}\label{fig:f2}}
	\caption{Bifurcation diagrams. Grey circles indicate stable periodic orbits while black circles indicate unstable orbits. }
\end{figure}
\subsubsection{Slow System}
For the slow system $\tau =\epsilon t$ and thus the slow system dynamics are given by 
\begin{eqnarray*}
	\epsilon \frac{dT}{d\tau} &=& \Lambda- \mu T -\alpha P^2T\\
	\epsilon \frac{dP}{d\tau} &=& \alpha P^2T - \gamma P-\delta P W\\
	\frac{dW}{d\tau} &=&\kappa P-c W.
\end{eqnarray*}
On the singular limit, the system reduces to
\begin{eqnarray}\label{10}
0 &=& \Lambda - \mu T -\alpha P^2T \nonumber\\
0 &=& \alpha P^2T - \gamma P-\delta P W \\
\frac{dW}{d\tau} &=&\kappa P-c W\nonumber.
\end{eqnarray}
We notice that the first two equations give us the slow manifold which consists of two branches that can be expressed as
$W=\phi(P)= \frac{-\gamma}{\delta}+\frac{\alpha \Lambda P}{\delta(\alpha P^2+\mu)}$   and  $P=0$. We plot this slow manifold in Figure \ref{fig:f3}. The upper branch of the slow manifold follows the fate of an infected individual while the lower branch focuses on the recovery. To emphasize the focus on the infected part of the slow manifold, we define $\omega$ to be a function of the immune response, which we refer to as the immune status. The infection begins at the time point where $W=\omega^*_0$ and continues until it reaches the tip of the manifold, we assume $W= \omega_0$ at this point, where due to the slow-fast system there is a jump into the recovered branch $(P=0)$ of the manifold. At this point, the pathogen is cleared and the infected individual recovers.\\ From the last equation in (\ref{10}) we express the nullcline of $W$ as $W=\frac{\kappa P}{c}$. We add this nullcline to Figure \ref{fig:f3}. We observe that a minimum pathogen threshold is required to activate an immune response. 
\begin{figure}[h]
	\centering
	\includegraphics{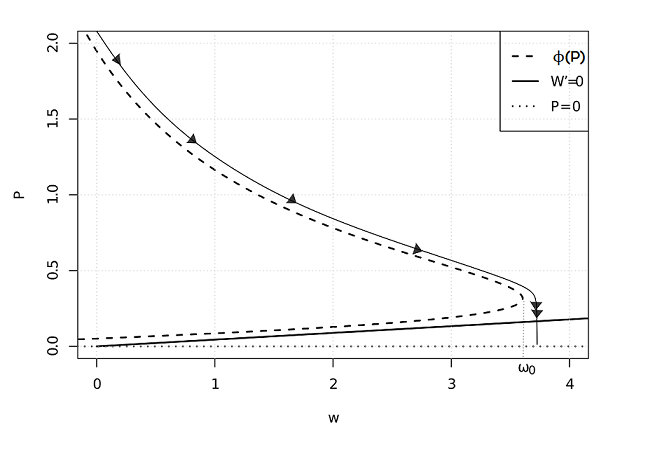}
	\caption{Trajectories of the system}\label{fig:f3}
\end{figure}
\\We now consider the immune status $\omega$, which we have defined above, to be a physiological variable that changes with respect to time. We describe this change by the ODE
\begin{equation*}
\omega'=g(\omega)  \qquad  \omega(0)=\omega^*_0 
\end{equation*} where $\omega^*_0$ is the initial immune status and $g(\omega)$  is the individual immune growth rate given by $g(\omega)=\phi^{-1}(W)-cW$. Additionally, we take $g(0)>0$ such that the immune status increases with time. We note that a single infected individual has immune status  $\omega^*_0$ $(\omega_0^*=0)$ at the start of an infection and recovers at the point where the immune status is $\omega_0$. From a single infected individual, we scale up the infection to the population level where we structure the density of the infected population by the immune status. 
\section{Between-Host Model}\label{sec:3}
The use of physiologically structured models to study populations has been advanced by \citet{metz2014dynamics, diekmann2007physiologically, cushing1998introduction, auger2008structured}. These physiological variables include age, size, immunity status and many more. Epidemic models structuring the population by immunological variables have been explored in  \citet{angulo2013sir, martcheva2006epidemic}. Using the previous work as a basis, we formulate a physiological structured model to represent the epidemiological dynamics of the disease.  The model is size-structured with the physiological variable being the immune status of individuals. The density of the infected population with immune status $\omega$ at time t is given as I (t, $\omega$). The model also takes into account the indirect and direct transmission pathways of the pathogen as well as the role of the environment in the disease dynamics. The between-host model is given as 
\begin{eqnarray}\label{3.1}
\frac{d S(t)}{dt}&=& r-\mu_1 S(t)-  S(t)\int_{0}^{\omega_0}\beta_h P(\omega) I (t,\omega)d\omega -\beta_e S(t)B(t)\nonumber\\&&+ \rho V(t) \nonumber \\
\partial_t I(t, \omega)+\partial_\omega (g(\omega)I(t, \omega))&=&-\mu_2 (\omega) I(t, \omega)\nonumber \\
g(0)I(t,0)&=& S(t)\int_{0}^{\omega_0}\beta_h P(\omega) I(t, \omega)d\omega +\beta_e S(t)B(t) \\
\frac{d V(t)}{dt}&=& g(\omega_0)I(t, \omega_0) -\rho V(t)- \mu_3 V(t) \nonumber \\
\frac{d B(t)}{dt}&=&\int_{0}^{\omega_0} \xi (\omega)P(\omega) I(t, \omega)d\omega -\sigma B(t)\nonumber
\end{eqnarray}
with the initial conditions $S(0)=S_0$,  $V(0)=V_0$, $B(0)=B_0$ and $I(0,\omega)=\phi(\omega)$. The variables S, I, R denote the density of susceptible, infected and immune individuals respectively. The variable B represents the bacterial concentration in the environment. The parameter  $r$ is the rate of recruitment of susceptible individuals, $\mu_1$, $\mu_2$ and $\mu_3$ are the natural death rates of the susceptible, infected and immune individuals respectively. $\beta_h$ is the direct transmission rate, $\beta_e$ is the indirect transmission rate, $g(\omega_0)$ is the recovery rate (due to growth of immunity) of infected individuals,  $\rho$ is the rate of loss of immunity,  $\xi $ is the shedding rate of bacteria back to the environment by infected hosts and $\sigma$ is the death rate of the bacteria. The infectivity of an infectious person is taken to be dependant on the within-host pathogen load $P(\omega)$, which is also considered to be dependant on the immune status. 
\subsection{Existence of Solutions}
\begin{prop}\label{prop 3.1}
	The solution of the PDE in the system (\ref{3.1}) with the initial and boundary conditions is given by
	\begin{equation}
	I(t,\omega)= \begin{cases}\phi (G^{-1}(G(\omega)-t))\frac{g(G^{-1}(G(\omega)-t))}{g(\omega)}e^{-\int_{G^{-1}(G(\omega)-t)}^{\omega}\frac{\mu_2(\tau)}{g(\tau)}d\tau}, \nonumber \\
	H(t-G(\omega))\frac{1}{g(\omega)}e^{-\int_{0}^{\omega}\frac{\mu_2(\tau)}{g(\tau)}d\tau}.
	\end{cases}
	\end{equation}
\end{prop}
The existence and uniqueness of solutions can be shown using the standard methods for establishing existence and uniqueness in physiologically structured epidemic models \citep{ kim1995mathematical, calsina1995model}.
\subsection{Basic Reproduction Number and Stability of the DFE}
The reproduction number $\mathcal R_0$ is defined as the expected number of secondary infections produced when a single infected person is introduced into a purely susceptible population \citep{diekmann1990definition}. The disease-free equilibrium (DFE) of system (\ref{3.1})  always exists and is given by  $\mathcal E_0 \ = (S^*_0,0,0,0)$ where $S^*_0=\frac{r}{\mu_1}$. To check for stability of the DFE we linearize system (\ref{3.1}) around the disease-free equilibrium and in the process, we also find the threshold condition for the spread of the disease which is considered to be the reproduction number.

\begin{theorem}
	The disease free equilibrium is locally asymptotically stable when $\mathcal R_0<1$ and unstable if $\mathcal R_0>1$, where
	\begin{equation*}
	\mathcal R_0= \frac{r\beta_h}{\mu_1}\int_{0}^{\omega_0} \frac{P(\omega)}{g(\omega)}e^{-\int_{0}^{\omega}\frac{\mu_2(\tau)}{g(\tau)}d\tau
	} d\omega
	+ \frac{r\beta_e}{\mu_1}\int_{0}^{\omega_0}\frac{\xi (\omega)}{\sigma}  \frac{P(\omega)}{g(\omega)}e^{-\int_{0}^{\omega}\frac{\mu_2(\tau)}{g(\tau)}d\tau}d\omega.
	\end{equation*}
\end{theorem}
\begin{proof}
	We let $S(t)= S^{*}_0+x(t)$, $I(t,\omega)=I_1(t,\omega)$, $V=y(t)$ and $B=z(t)$. Substituting these perturbed expressions into (\ref{3.1}) and simplifying gives us the linearized system
	\begin{eqnarray}\label{3.2}
	\frac{d x(t)}{dt} &=& -\mu_1 x(t)-  S^{*}_0\int_{0}^{\omega_0}\beta_h P(\omega) I_1 (t,\omega )d\omega 
	-\beta_e S^{*}_0z(t)\nonumber \\&&+ \rho y(t)  \nonumber \\
	\partial_t I_1(t, \omega)+\partial_\omega (g(\omega)I_1(t, \omega))&=&-\mu_2 (\omega) I_1(t, \omega)\nonumber \\
	g(0)I_1(t,0)&=& S^{*}_0\int_{0}^{\omega_0}\beta_h P(\omega) I_1(t, \omega)d\omega
	+\beta_e S^{*}_0z(t) \\
	\frac{d y(t)}{dt}&=& g(\omega_0)I_1(t, \omega_0) - \rho y(t)- \mu_3 y(t)\nonumber \\
	\frac{d z(t)}{dt}&=&\int_{0}^{\omega_0} \xi (\omega)P(\omega) I_1(t, \omega)d\omega -\sigma z(t)\nonumber.
	\end{eqnarray}
	We look for solutions of the form $x(t)=\bar{x}e^{\lambda t}$, $I_1(t,w)=\bar{I_1}(\omega)e^{\lambda t}$, $y(t)=\bar{y}e^{\lambda t}$, $z(t)=\bar{z}e^{\lambda t}$.
	Subsituting the appropriate form in system (\ref{3.2}) gives us the eigen value problem
	\begin{eqnarray}\label{3.3}
	\lambda \bar{x} &=& -\mu_1 \bar{x}-  S^{*}_0\int_{0}^{\omega_0}\beta_h P(\omega) \bar{I_1} (\omega )d\omega 
	-\beta_e S^{*}_0\bar{z}+ \rho\bar{y}  \nonumber \\
	\partial_\omega (g(\omega)\bar{I_1}( \omega))&=&-(\mu_2 (\omega)+\lambda )\bar{I_1}( \omega)\nonumber \\
	g(0)\bar{I_1}(0)&=& S^{*}_0\int_{0}^{\omega_0}\beta_h P(\omega) \bar{I_1}(\omega)d\omega
	+\beta_e S^{*}_0\bar{z} \\
	\lambda \bar{y}&=& g(\omega_0)\bar{I_1}( \omega_0) - \rho \bar{y}- \mu_3 \bar{y} \nonumber \\
	\lambda \bar{z}&=&\int_{0}^{\omega_0} \xi (\omega)P(\omega) \bar{I_1}( \omega)d\omega -\sigma \bar{z} \nonumber.
	\end{eqnarray}
	Solving for the second equation of system (\ref{3.3}) gives us
	\begin{eqnarray*}
		\bar{I_1}(\omega)&=&\frac{\bar{I_1}(0)g(0)}{g(\omega)}e^{-\int_{0}^{\omega}\frac{\mu_2(\tau)+\lambda}{g(\tau)}d\tau}.
	\end{eqnarray*}
	Using the fifth equation of (\ref{3.3}) and $\bar{I_1}(\omega)$ we can express $\bar{z}$ as
	\begin{equation*}
	\bar{z}=\frac{\bar{I_1}(0)g(0)}{\lambda+\sigma}\int_{0}^{\omega_0} \xi (\omega)\frac{P(\omega)}{g(\omega)}e^{-\int_{0}^{\omega}\frac{\mu_2(\tau)+\lambda}{g(\tau)}d\tau}d\omega.
	\end{equation*}
	Substituting this expression of $\bar{z}$ and $\bar{I_1}(\omega)$ into the third equation of (3.3) gives us
	\begin{eqnarray}\label{3.4}
	 g(0)\bar{I_1}(0)&=&S^{*}_0\bar{I_1}(0)g(0)\bigg[\int_{0}^{\omega_0}\beta_h \frac{P(\omega)}{g(\omega)}e^{-\int_{0}^{\omega}\frac{\mu_2(\tau)+\lambda}{g(\tau)}d\tau} d\omega
	\nonumber\\&&+\frac{\beta_e}{\lambda+\sigma}\int_{0}^{\omega_0} \xi (\omega)\frac{P(\omega)}{g(\omega)}e^{-\int_{0}^{\omega}\frac{\mu_2(\tau)+\lambda}{g(\tau)}d\tau}d\omega\bigg].
	\end{eqnarray}
	Respectively, we obtain the characteristic equation $G(\lambda)=1$ with
	\begin{equation}\label{3.5}
	G(\lambda)= S^{*}_0\bigg[\int_{0}^{\omega_0}\beta_h \frac{P(\omega)}{g(\omega)}e^{-\int_{0}^{\omega}\frac{\mu_2(\tau)+\lambda}{g(\tau)}d\tau} d\omega
	+ \frac{\beta_e}{\lambda+\sigma}\int_{0}^{\omega_0} \xi (\omega)\frac{P(\omega)}{g(\omega)}e^{-\int_{0}^{\omega}\frac{\mu_2(\tau)+\lambda}{g(\tau)}d\tau}d\omega\bigg].
	\end{equation}
	For the DFE to be stable, the roots of the characteristic equation should have negative real parts otherwise, it's unstable. We use the approach in \citet{martcheva2015introduction} to check for the stability of the DFE.\\
	A non-zero solution to (\ref{3.4}) exists if only there is a number $\lambda \in \mathbb{R}$ such that $G(\lambda) =1$.
	Differentiating Equation (\ref{3.5}) with respect to $\lambda$ yields $G'(\lambda)<0$ and thus $G(\lambda)$ is a strictly decreasing function, additionally, $\lim_{\lambda\to\infty} G(\lambda)=0$. If $ \hat \lambda $ is a unique real solution of (\ref{3.5}) then $\hat \lambda> 0 $ provided $G(0)> 1 $ and $ \hat \lambda< 0 $ provided $G(0) < 1$. \\
	If we let $H=S^{*}_0\beta_h\frac{P(\omega)}{g(\omega)}e^{-\int_{0}^{\omega} \frac{\mu_2(\tau)}{g(\tau)}d\tau}$ and $J= S^{*}_0\beta_e \sigma(\omega)\frac{P(\omega)}{g(\omega)}e^{-\int_{0}^{\omega} \frac{\mu_2(\tau)}{g(\tau)}d\tau}$ we can express $G(\lambda)$ as
	\begin{equation*}
	G(\lambda)= \int_{0}^{\omega_0}H  e^{-\int_{0}^{\omega}\frac{\lambda}{g(\tau)}d\tau} d\omega
	+ \frac{1}{\lambda+\sigma}\int_{0}^{\omega_0} J  e^{-\int_{0}^{\omega}\frac{\lambda}{g(\tau)}d\tau}d\omega
	\end{equation*}
	Suppose $G(0)<1$ and $\lambda=a\pm bi$ is a complex solution to equation (\ref{3.5}) with $a\geq 0$. Then
	\begin{eqnarray}
	\mid G(\lambda)\mid
	&\le & \int_{0}^{\omega_0}H  e^{-\int_{0}^{\omega}\frac{a}{g(\tau)}d\tau}  d\omega
	+ \frac{1}{(a+\sigma) }\int_{0}^{\omega_0} J  \ e^{-\int_{0}^{\omega}\frac{a}{g(\tau)}d\tau} d\omega \nonumber = G(a)\leq G(0) <1. \nonumber 
	\end{eqnarray}
	It follows then that equation (\ref{3.5}) has a complex solution $\lambda=a\pm ib$ if $a<0$ and that solution must always have a negative real part. 
	$G(0)=1$ is considered to be the threshold for the stability of the disease-free  equilibrium and is called the basic reproduction number, that is $G(0)=\mathcal R_0$ where
	\begin{equation*}
	\mathcal R_0= \frac{r\beta_h}{\mu_1}\int_{0}^{\omega_0} \frac{P(\omega)}{g(\omega)}e^{-\int_{0}^{\omega}\frac{\mu_2(\tau)}{g(\tau)}d\tau
	} d\omega
	+ \frac{r\beta_e}{\mu_1}\int_{0}^{\omega_0}\frac{\xi (\omega)}{\sigma}  \frac{P(\omega)}{g(\omega)}e^{-\int_{0}^{\omega}\frac{\mu_2(\tau)}{g(\tau)}d\tau}d\omega.
	\end{equation*}
	The disease-free equilibrium is therefore locally asymptotically stable if $\mathcal R_0 <1$ and unstable otherwise.\qed
\end{proof}
\subsection{Existence of   the Endemic Equilibrium}
\begin{prop}
	A unique positive endemic equilibrium of system (\ref{3.1}) given by $\mathcal{E^*}=(S^*,I^*(\omega),V^*, B^*)$ exists if $\mathcal R_0>1$.
\end{prop}
\begin{proof}
	To find the endemic equilibrium we solve the system
	\begin{eqnarray}\label{3.6}
	0&=& r- \mu_1 S^*-  S^*\int_{0}^{\omega_0}\beta_h P(\omega) I^* (\omega)d\omega -\beta_e S^*B^*+ \rho V^* \nonumber \\
	\partial_\omega(g(\omega) I^*( \omega))&=&-\mu_2 (\omega) I^*( \omega) \nonumber \\
	g(0)I^*(0)&=& S^*\int_{0}^{\omega_0}\beta_h P(\omega) I^*(\omega)d\omega +\beta_e S^*B^*  \\
	0&=& g(\omega_0)I^*(\omega_0) -\rho V^*- \mu_3 V^* \nonumber \\
	0&=&\int_{0}^{\omega_0} \xi (\omega)P(\omega) I^*(\omega)d\omega -\sigma B^*\nonumber.
	\end{eqnarray}
	Solving for the second equation in system (3.6) gives	$I^*(\omega)=\frac{I^*(0)g(0)}{g(\omega)}e^{-\int_{0}^{\omega}\frac{\mu_2(\tau)}{g(\tau)}d\tau}$.	If we let $\frac{1}{g(\omega)} e^{-\int_{0}^{\omega}\frac{\mu_2(\tau)}{g(\tau)}d\tau}$ be $\pi(\omega)$ then 
	\begin{equation*}
	I^*(\omega)=I^*(0)g(0)\pi(\omega).
	\end{equation*}
	Substituting $I^*(\omega)$ to the fourth and fifth equations in system (\ref{3.6}) gives
	\begin{equation*}
	B^*=I^*(0)g(0)\int_{0}^{\omega_0} \frac{\xi(\omega)}{\sigma}P(\omega)\pi(\omega) d\omega, \qquad V^*=\frac{g(\omega_0)I^*(0)g(0) \pi (\omega_0)}{(\rho +\mu_3)}.
	\end{equation*} 
	Substituting $I^*(\omega)$ and $B^*$ in the third equation of system (\ref{3.6}) yields
	\begin{equation*}
	S^*= \frac{1}{\int_{0}^{\omega_0}\beta_h P(\omega)\pi(\omega) d\omega +\int_{0}^{\omega_0}\beta_e \frac{\xi(\omega)}{\sigma}P(\omega)\pi(\omega) d\omega }.
	\end{equation*}
	The first equation in 
	system (\ref{3.6}) can be rewritten as
	\begin{equation}\label{3.7}
	r- \mu_1 S^*- g(0)I^*(0) + \rho V^* = 0 .
	\end{equation}
	Rewriting $S^*$ in terms of $\mathcal R_0$ and substituting it and $V^*$ into equation (\ref{3.7} ) yields
	\begin{equation*}
	r- \frac{r}{\mathcal R_0 }-g(0)I^*(0)+I^*(0)g(0)\frac{\rho g(\omega_0) \pi (\omega_0)}{(\rho +\mu_3)}=0.
	\end{equation*}
	Making $I^*(0)$ the subject yields
	$I^*(0)=\frac{r (1-\frac{1}{\mathcal R_0})}{g(0)(1-\frac{\rho g(\omega_0) \pi (\omega_0)}{(\rho +\mu_3)})}$. Substituting $I^*(0)$ back into the expression of $I^*(\omega)$ gives us 
	\begin{equation}\label{3.8}
	I^*(\omega)=\frac{r (1-\frac{1}{\mathcal R_0})}{(1-\frac{\rho g(\omega_0) \pi (\omega_0)}{(\rho +\mu_3)})}\pi (\omega). 
	\end{equation}
	Since $\frac{\rho g(\omega_0) \pi (\omega_0)}{(\rho +\mu_3)}=\frac{\rho e^{-\int_{0}^{\omega}\frac{\mu_2(\tau)}{g(\tau)}d\tau}}{(\rho +\mu_3)}<1$, we need to have $\mathcal R_0 >1$ to get a positive $I^*(\omega)$, thus the endemic equilibrium $\mathcal{E^*}=(S^*,I^*(\omega),V^*, B^*)$ exists only if  $\mathcal R_0 >1$. \qed
\end{proof}	
\subsection{Local Stability of the Endemic Equilibrium}
We assume that $\mathcal R_0>1$ and linearize system (\ref{3.1}) around the endemic equilibrium. We let  $S(t)= S^{*}+x(t)$, $I(t,\omega)=I^{*}(\omega)+I_1(t,\omega)$, $V= V^{*}+y(t)$ and $B=B^{*}+z(t)$ and substitute these expressions in (\ref{3.1}) to get the linearized system 
\begin{eqnarray}\label{3.9}
\frac{d x(t)}{dt} &=& - \mu_1 x(t)-  S^{*}\int_{0}^{\omega_0}\beta_h P(\omega)I_1(t, \omega) d\omega+\rho y(t)
\nonumber\\& &-x(t)\int_{0}^{\omega_0}\beta_h P(\omega)I^{*}(\omega)d(\omega) -\beta_e S^{*}z(t)-\beta_e B^{*}x(t)   \nonumber \\
\partial_t I_1(t, \omega) +\partial_\omega (g(\omega)I_1(t, \omega))&=&-\mu_2 (\omega)I_1(t, \omega) \nonumber \\
g(0)I_1(t,0)&=&  S^{*}\int_{0}^{\omega_0}\beta_h P(\omega)I_1(t, \omega) d\omega +\beta_e S^{*}z(t) \nonumber \\& &+ x(t)\int_{0}^{\omega_0}\beta_h P(\omega)I^{*}(\omega)d(\omega) +\beta_e B^{*}x(t) \\
\frac{d y(t)}{dt}&=& g(\omega_0)I_1(t, \omega_0) -\rho y(t) -\mu_3y(t) \nonumber \\
\frac{d z(t)}{dt}&=&\int_{0}^{\omega_0} \xi (\omega)P(\omega) I_1(t, \omega)d\omega-\sigma z(t)\nonumber.
\end{eqnarray}
We look for solutions of the form $x(t)=x e^{\lambda t}$, $I_1(t,w)=I_1(\omega)e^{\lambda t}$, $y(t)=y e^{\lambda t}$, $z(t)=z e^{\lambda t}$.
Subsituting the appropriate form in system (\ref{3.9}) yields the eigen value problem
\begin{eqnarray}\label{3.10}
\lambda x &=& - \mu_1 x-  S^{*}\int_{0}^{\omega_0}\beta_h P(\omega)I_1(\omega) d\omega
-x\int_{0}^{\omega_0}\beta_h P(\omega)I^{*}(\omega)d(\omega) \nonumber\\& & -\beta_e S^{*}z-\beta_e B^{*}x +\rho y  \nonumber \\
\partial_\omega (g(\omega)I_1( \omega))&=&-(\mu_2 (\omega)+\lambda )I_1(\omega) \nonumber \\
g(0)I_1(0)&=&  S^{*}\int_{0}^{\omega_0}\beta_h P(\omega)I_1(\omega) d\omega + x\int_{0}^{\omega_0}\beta_h P(\omega)I^{*}(\omega)d(\omega) \\& & +\beta_e S^{*}z+\beta_e B^{*}x \nonumber \\
\lambda y&=& g(\omega_0)I_1( \omega_0) -\rho y -\mu_3y \nonumber \\
\lambda z&=&\int_{0}^{\omega_0} \xi (\omega)P(\omega) I_1( \omega)d\omega-\sigma z\nonumber.
\end{eqnarray}
Solving the second equation in system (\ref{3.10}) gives us
\begin{equation}\label{3.11}
I_1(\omega)=I_1(0)g(0)\pi_1 (\omega)e^{-\int_{0}^{\omega}\frac{\lambda}{g(\tau)}d\tau},
\end{equation}
where $\pi_1(\omega)=\frac{1}{g(\omega)}e^{-\int_{0}^{\omega}\frac{\mu_2(\tau)}{g(\tau)}d\tau}$.
Adding the first and third equation in (\ref{3.9}) gives us $x=\frac{\rho y-g(0)I_1(0)}{\lambda+ \mu_1}$.
Solving for $y$ and substituting its expression in $x$ and $z$ gives
\small\begin{equation*} y=\frac{g(0)I_1(0)g(\omega_0 )\pi_1(\omega_0)e^{-\int_{0}^{\omega}\frac{\lambda}{g(\tau)}d\tau}}{\lambda+ \rho+ \mu_3},\quad x=\frac{\frac{\rho g(0)I_1(0)g(\omega_0 )\pi_1(\omega_0)e^{-\int_{0}^{\omega}\frac{\lambda}{g(\tau)}d\tau}}{\lambda+ \rho+ \mu_3}  -g(0)I_1(0)}{\lambda+ \mu_1}
\end{equation*}\normalsize
\begin{equation*}
z=\frac{I_1(0)g(0)\int_{0}^{\infty} \xi (\omega)P(\omega) \pi_1 (\omega)e^{-\int_{0}^{\omega}\frac{\lambda}{g(\tau)}d\tau}}{\lambda+\sigma}.
\end{equation*}
Letting $K=\int_{0}^{\omega_0}\beta_h P(\omega)I^{*}(\omega)d(\omega)$ and substituting $x, y,$ and $z$ in the third equation of (\ref{3.9}) we get the characteristic equation
\begin{eqnarray}\label{ex.2}	
1&=& S^{*}\int_{0}^{\omega_0}\beta_h P(\omega)\pi_1 (\omega)e^{-\int_{0}^{\omega}\frac{\lambda}{g(\tau)}d\tau} d\omega +\bigg[\frac{\frac{\rho g(\omega_0 )\pi_1(\omega_0)e^{-\int_{0}^{\omega}\frac{\lambda}{g(\tau)}d\tau}}{\lambda+ \rho+ \mu_3}  -1}{\lambda+ \mu_1}\bigg] K \\ &&+\beta_e S^{*}\frac{\int_{0}^{\omega_0} \xi (\omega)P(\omega) \pi_1 (\omega)e^{-\int_{0}^{\omega}\frac{\lambda}{g(\tau)}d\tau} d\omega}{\lambda+\sigma} +\beta_e B^{*}\bigg[\frac{\frac{\rho g(\omega_0 )\pi_1(\omega_0)e^{-\int_{0}^{\omega}\frac{\lambda}{g(\tau)}d\tau}}{\lambda+ \rho+ \mu_3}  -1}{\lambda+ \mu_1}\bigg]\nonumber .
\end{eqnarray}
For a single cholera infection, the loss of immunity only plays a minor role. We therefore focus on the case $\rho=0$.
\begin{theorem}
	Given no loss of immunity, the endemic equilibrium is locally asymptotically stable whenever $\beta_e=0$ and $g(\omega)=1$. 
\end{theorem}
\begin{proof}
	The characteristic equation (\ref{ex.2}) reduces to
	\begin{equation}	
	1= S^{*}\int_{0}^{\omega_0}\beta_h P(\omega)\hat \pi_1 (\omega)e^{-\lambda\omega}  d\omega -\frac{K}{\lambda+ \mu_1}   \nonumber \qquad where\qquad \hat \pi_1(\omega)=e^{-\int_{0}^{\omega}\mu_2(\tau)d\tau}.
	\end{equation}
	\begin{equation}\label{ex.1}
	\frac{\lambda+ \mu_1+K}{\lambda+ \mu_1}= S^{*}\int_{0}^{\omega_0}\beta_h P(\omega)\hat \pi_1 (\omega)e^{-\lambda\omega}  d\omega. 
	\end{equation}
	If we let $\lambda=a+ib$ and assume that $a\ge 0$, then for $\mathbb{R}(\lambda)\ge 0$ the left hand side of equation (\ref{ex.1}) gives
	\begin{equation*}
	\mid \frac{\lambda+ \mu_1+K}{\lambda+ \mu_1}\mid>1, 
	\end{equation*}
	while the right hand side yields
	\begin{eqnarray*}
		\mid S^{*}\int_{0}^{\omega_0}\beta_h P(\omega)\hat \pi_1 (\omega)e^{-\lambda\omega}  d\omega\mid &\le& S^{*}\int_{0}^{\omega_0}\beta_h P(\omega)\hat \pi_1 (\omega)\mid e^{-\lambda\omega}\mid  d\omega\\
		&\le& S^{*}\int_{0}^{\omega_0}\beta_h P(\omega)\hat \pi_1 (\omega) e^{-a\omega}  d\omega\\
		&\le& S^{*}\int_{0}^{\omega_0}\beta_h P(\omega)\hat \pi_1 (\omega)=1. 
	\end{eqnarray*}	
	Thus, given $\lambda$ with $\mathbb{R}(\lambda)\ge0$, the left side of equation (\ref{ex.1}) is strictly greater than one while the right side of equation (\ref{ex.1}) is strictly less than one, which is a contradiction. Therefore, any $\lambda$ with non-negative real parts does not satisfy the characteristic equation and the endemic equilibrium is locally asymptotically stable.\qed 
\end{proof}
\subsection{Global Stability of the Endemic Equilibrium}
\citet{meehan2019global} describes a susceptible class experiencing a force of infection $F(t)$ as
\begin{eqnarray}
\frac{d S(t)}{dt}&=& \lambda- \mu_1 S(t)-  S(t)F(t)\nonumber\\
F(t)&=&\int_{0}^{\bar \tau}A(\tau) S(t-\tau)F(t-\tau)d\tau
\end{eqnarray}
where $A(\tau)$ is the contribution of  individuals infected for time $\tau$  to the force of infection and the infectivity kernel $A\ge 0$. He assumes that the maximal age of infection $\bar \tau<\infty$ and that the infection confers permanent immunity. Using Lyapunov functionals, he concludes from the integral-form, with compact support of the
integral kernel, the global stability of the DFE when $\mathcal{R}_0\leq1$ and the global stability of the endemic equilibrium when $\mathcal{R}_0>1$. We aim to formulate system (\ref{3.1}) in terms of his results to establish global stability.
\begin{theorem} 
	Given no loss of immunity, the endemic equilibrium is globally asymptotically stable if $\mathcal{R}_0>1$.
\end{theorem}
\begin{proof}
	For technical reasons, we restructure system (\ref{3.1}) such that the environmental bacteria have a maximal age, that is,  $B(t)=\int_{0}^{\bar a} B(t,a)da$. We consider this to be more realistic and biologically meaningful. 
	We define the force of infection 
	$F(t)=\int_{0}^{\omega_0}\beta_h P(\omega) I (t,\omega)d\omega +\int_{0}^{\bar a}\beta_e B(t,a)da$ such that (\ref{3.1}) becomes
	\begin{eqnarray}\label{3.14}
	\frac{d S(t)}{dt}&=& r- \mu_1 S(t)-  S(t)F(t)\nonumber\\
	\partial_t I(t,\omega)+\partial_\omega(g(\omega)I(t, \omega))&=&-\mu_2 (\omega) I(t, \omega)\nonumber \\
	g(0)I(t,0)&=&S(t)F(t)\\
	(\partial_t +\partial_a) B(t,a)&=&-\sigma B(t,a)\nonumber \\
	B(t,0)&=&\int_{0}^{\omega_0} \xi (\omega)P(\omega) I(t, \omega)d\omega\nonumber. 
	\end{eqnarray}
	We aim to rewrite $F(t)$ as 	$F(t)=\int_{0}^{\bar \tau}A(\tau) S(t-\tau)F(t-\tau)d\tau$. From proposition (\ref{prop 3.1}) 
	\begin{equation*}
	I(t,\omega)=\begin{cases} S(t-G(\omega))F(t-G(\omega))\frac{1}{g(\omega)}e^{-\int_{0}^{\omega}\frac{\mu_2(\omega')}{g(\omega')}d\omega'}&\omega\le\omega_0\\0   & \omega>\omega_0.
	\end{cases}
	\end{equation*}
	Solving for $B(t,a)$ gives
	\begin{eqnarray*}
		B(t,a)&=&e^{-\sigma a}\int_{0}^{\omega_0} \xi(\omega)P(\omega)I(t-a,\omega)d\omega\\
		 \int_{0}^{\bar a}B(t,a)da&=&
		 \int_{0}^{\bar a}e^{-\sigma a}\int_{0}^{\omega_0} \frac{\xi(\omega)P(\omega)}{g(\omega)}S(t-a-G(\omega)) F(t-a-G(\omega))\\&&\qquad \cdot e^{-\int_{0}^{\omega}\frac{\mu_2(\omega')}{g(\omega')}d\omega'}d\omega da.
	\end{eqnarray*}
	If we let $a+G(\omega)=\theta$ then
	\begin{eqnarray*}
		\int_{0}^{\bar a}B(t,a)da&=&\int_{0}^{\bar a}e^{-\sigma a}\int_{a+G(0)}^{a+G(\omega_0)} \xi(G^{-1}(\theta-a))P(G^{-1}(\theta-a))S(t-\theta)\\&&\qquad \cdot F(t-\theta)e^{-\int_{0}^{G^{-1}(\theta-a)}\mu_2(\omega')d\omega'}d\theta da.
	\end{eqnarray*}
	The force of infection becomes
	\begin{eqnarray*}
	F(t)&=&\int_{0}^{\omega_0}\beta_h P(\omega) S(t-\omega)F(t-\omega)e^{-\int_{0}^{\omega}\mu_2(\omega')d\omega'}d\omega\\&&+\int_{0}^{\bar a}\beta_ee^{-\sigma a}\int_{a+G(0)}^{a+G(\omega_0)} \xi(G^{-1}(\theta-a))P(G^{-1}(\theta-a))S(t-\theta)\\&&\qquad \cdot F(t-\theta)e^{-\int_{0}^{G^{-1}(\theta-a)}\mu_2(\omega')d\omega'}d\theta da.
	\end{eqnarray*}
	Defining
	\begin{eqnarray*}
		K_h(\omega)&=&\beta_h P(\omega) e^{-\int_{0}^{\omega}\mu_2(\omega')d\omega'}\\
		K_e(\theta)&=
		&\
	\begin{cases}\int_{0}^{\theta}\beta_ee^{-\sigma a} \xi(G^{-1}(\theta-a))P(G^{-1}(\theta-a))S(t-\theta)F(t-\theta)\\\qquad \cdot e^{-\int_{0}^{G^{-1}(\theta-a)}\mu_2(\omega')d\omega'}d\theta \quad if \quad \theta<G(\omega_0)\\
			\int_{\theta-G(\omega_0)}^{\theta}\beta_ee^{-\sigma a} \xi(G^{-1}(\theta-a))P(G^{-1}(\theta-a))S(t-\theta)F(t-\theta)\\\qquad \cdot e^{-\int_{0}^{G^{-1}(\theta-a)}\mu_2(\omega')d\omega'}d\theta \quad if\quad G(\omega_0)\le \theta\le G(\omega_0)+\bar a,
		\end{cases}
	\end{eqnarray*}
	gives the renewal equation
	\begin{eqnarray*}
		F(t)&=&\int_{0}^{\omega_0}K_h(\omega) S(t-\omega)F(t-\omega)d\omega+\int_{0}^{\bar a+G(\omega_0)}K_e(\omega) S(t-\omega)F(t-\omega)d\omega\\
		F(t)&=&\int_{0}^{\omega_0}A(\omega) S(t-\omega)F(t-\omega)d\omega 
	\end{eqnarray*}
	where
	\begin{equation}
	A(\omega)= \begin{cases}K_h(\omega) & \omega\le\omega_0 \nonumber \\ K_e(\omega)&\omega\le \bar a+G(\omega_0). 
	\end{cases}
	\end{equation}
	The susceptible class is now described by the system
	\begin{eqnarray*}
		\frac{d S(t)}{dt}&=& r- \mu_1 S(t)-  S(t)F(t)\nonumber\\
		F(t)&=&\int_{0}^{\omega_0}A(\omega) S(t-\omega)F(t-\omega)d\omega 
	\end{eqnarray*}
	From the results of \citet{meehan2019global}, the endemic equilibrium of the above system has been shown to be globally asymptotically stable when $\mathcal{R}_0>1$.\qed
\end{proof}
\citet{brauer2013dynamics} also examines global stability of systems similar to (\ref{3.14}) in the case $g(\omega)=1$. Although the case of waning immunity has not been considered, analysis of models with waning immunity can be viewed in \citet{barbarossa2015immuno,nakata2014stability}.
\section{Discussion}\label{sec:4}
In this paper, we have developed an immuno-epidemiological model that links the within-host and between-host dynamics of cholera. We have introduced the first attempt, to the best of our knowledge, to structure the epidemic model of the disease using the immune status, which is a function derived from our within-host system.\\ The immunological model follows the fate of a single infected individual where we distinguish the pathogen dynamics from the dynamics of the immune response using timescales. Furthermore, we express the incidence rate to be quadratic in P to emphasize that higher pathogen densities are required in the growth of the pathogen. Using time scale methods we have conducted a thorough analysis of our model. The result of the bifurcation analysis reveals the necessary conditions for the occurrence of a saddle-node and Hopf bifurcation. There's also a possibility of the occurrence of a homoclinic bifurcation that would eliminate the periodic orbits. Subsequently, we have found that a minimum pathogen load is required to activate an immune response. Unlike other within-host cholera models, our modelling approach allows for the possibility of recovery, through the clearance of the pathogen, after a finite period.\\
We use a size-structured model to represent our between-host dynamics. The immune status is the structuring variable, which is an important aspect in terms of its role in the contraction of the disease.  We further linked the two models using the pathogen load and considered the direct and indirect transmission pathways of the disease. We derived the reproduction number and established the conditions for the stability of the DFE. We found the basic reproduction number to be dependent on both direct and indirect transmission pathways of the disease. This emphasizes the need for control measures that target the reduction of transmission by both routes. For the DFE, the disease will be eradicated if $\mathcal{R}_0<1$ and persist otherwise. We showed that a unique endemic equilibrium exists when $\mathcal{R}_0>1$. Without loss of immunity, the endemic equilibrium was both locally and globally asymptotically stable.  \\ Although we have provided a new framework for modelling the dynamics of the disease, our model also has several limitations. Stability analysis of the endemic equilibrium focuses on the case of permanent immunity, therefore,  neglecting the effects of waning immunity. The explicit linkage of the two systems is still inadequate in terms of embedding the within-host dynamics to the population dynamics of the disease. 
In our future work, we intend to provide better ways of connecting the within-host dynamics to the population dynamics of the disease by formulating an integrated model from which we can derive both our within-host and between-host dynamical systems.

\begin{acknowledgements}
The author thanks Johannes M{\"u}ller for his helpful discussions.
\end{acknowledgements}

 \section*{Conflict of interest}
None

\bibliographystyle{spbasic}      % basic style, author-year citations
\bibliography{references}   % name your BibTeX data base

%% Non-BibTeX users please use
%\begin{thebibliography}{}
%%
%% and use \bibitem to create references. Consult the Instructions
%% for authors for reference list style.
%\end{thebibliography}

\end{document}